\documentclass[11pt, amsfonts]{amsart}

\usepackage{amsmath,amssymb,amsthm,enumerate,comment}
\usepackage[all]{xy}
\usepackage{color}
\usepackage{graphicx}
\usepackage[T1]{fontenc}
\usepackage[dvipdfm,colorlinks=true]{hyperref}
\usepackage[colorlinks=true]{hyperref}
\SelectTips{eu}{12}

\theoremstyle{definition}
\newtheorem{theorem}{Theorem}[section]
\newtheorem{definition}[theorem]{\rm Definition}
\newtheorem{lemma}[theorem]{Lemma}
\newtheorem{proposition}[theorem]{Proposition}
\newtheorem{corollary}[theorem]{Corollary}
\newtheorem{example}[theorem]{\rm Example}
\newtheorem{remark}[theorem]{\rm Remark}
\newtheorem{notation}[theorem]{\rm Notation}

\makeatletter
\@addtoreset{equation}{section} 
\makeatother

\DeclareMathOperator{\Homeo}{Homeo}

\DeclareMathOperator{\id}{id}

\DeclareMathOperator{\Hom}{Hom}

\DeclareMathOperator{\grp}{grp}

\DeclareMathOperator{\ev}{ev}

\newcommand{\NN}{\mathbb{N}}
\newcommand{\RR}{\mathbb{R}}
\newcommand{\ZZ}{\mathbb{Z}}

\newcommand{\GG}{\Gamma}
\renewcommand{\gg}{\gamma}

\newcommand{\hG}{\widehat{G}_a}
\newcommand{\tG}{\widetilde{G}}

\newcommand{\hGb}{\widehat{G}_a^b}
\newcommand{\hGm}{\widehat{G}_{a, \mu}}
\newcommand{\hg}{\widehat{g}}
\newcommand{\tg}{\widetilde{g}}
\newcommand{\hh}{\widehat{h}}
\renewcommand{\th}{\widetilde{h}}
\newcommand{\G}{G_a}
\newcommand{\Gm}{G_{a,\mu}}
\newcommand{\Gb}{G_a^b}
\newcommand{\hX}{\widehat{X}_a}
\newcommand{\hx}{\widehat{x}}
\newcommand{\hy}{\widehat{y}}
\newcommand{\hf}{\widehat{f}}
\newcommand{\trot}{\widehat{\mathrm{r}}\mathrm{ot}}
\newcommand{\hgg}{\widehat{\gamma}}
\newcommand{\Z}{Z}
\newcommand{\A}{A}
\newcommand{\W}{W}
\newcommand{\tphi}{\widehat{\phi}}

\newcommand{\HHH}{\mathrm{H}}

\makeatletter
\@namedef{subjclassname@2020}{%
  \textup{2020} Mathematics Subject Classification}
\makeatother

\title[Translation numbers and a cocycle on a group of homeomorphisms]{Translation numbers for bundle automorphisms and a cocycle on a group of homeomorphisms}
\author{Shuhei Maruyama}
\address{Graduate School of Mathematics, Nagoya University, Japan}
\email{m17037h@math.nagoya-u.ac.jp}


\begin{document}

\begin{abstract}
  We introduce and study translation numbers for automorphisms of principal $\ZZ$-bundles and flat principal $\RR$-bundles.
  We use them to show a vanishing result of a characteristic class of foliated bundles and to detect undistortion elements in the group of bundle automorphisms.
  %Moreover, we show that these translation numbers detect undistortion elements in the group of bundle automorphisms.
  %We also give several examples that the characteristic class is non-zero.
\end{abstract}

\maketitle

%\linenumbers

\section{Introduction}\label{sec:intro}

Poincar\'{e}'s translation number is an invariant of one-dimensional dynamical systems, and it has been generalized to other cases,
%there have been many attempts to generalize it to other cases,
e.g., to higher dimensional manifolds.
Such generalizations include the asymptotic cycle \cite{MR88720} and the (homological) translation vectors (\cite{MR1053617}, \cite{MR1094554}, \cite{MR1325916}, \cite{MR1444450}), which are vector-valued invariants.
In \cite{MR2854098}, Gal and K\k{e}dra introduced another type of generalizations called the local rotation number (and the corresponding local translation number), which are contrastingly a numerical invariant.
In this paper, we consider the local translation number and introduce a measure theoretic counterpart which should be called the mean translation number.
%In this paper, we shall slightly modify the definition of the local translation number and introduce a measure theoretic counterpart which should be called the mean translation number.
%called the \textit{local translation number} and the \textit{mean translation number}.

%Here is the setting.
Let $\A$ be either the group $\ZZ$ of integers or the group $\RR^{\delta}$ of real numbers with the discrete topology.
For a path-connected topological space $X$ and a cohomology class $a \in \HHH^1(X;\A)$, let $\pi \colon \hX \to X$ denote a principal $\A$-bundle corresponding to $a$.
A homeomorphism $\hg \colon \hX \to \hX$ is called a \textit{bundle automorphism} if there exists a homeomorphism $g \colon X \to X$ such that $\pi \circ \hg = g \circ \pi$ holds.
Let $\hG$ denote the group of bundle automorphisms of $\hX \to X$.

For a bundle automorphism $\hg$, we define its local translation number $\trot_{x, \alpha}(\hg)$ and mean translation number $\trot_{\mu, \alpha}(\hg)$ in Section \ref{sec:local_trans_number} (our definition of $\trot_{x, \alpha}(\hg)$ slightly differ from one in \cite{MR2854098}; see Remark \ref{rem:differ_from_GK}).
Here $x$ is a point of $X$, $\mu$ is a $g$-invariant Borel probability measure on $X$, and $\alpha$ is a singular one-cocycle representing $a \in \HHH^1(X;\A)$.
As suggested by their symbols, these translation numbers depend on the choice of $x$, $\mu$, and $\alpha$.
They are given as a limit and an integration, and may not exist in some cases; see Definitions \ref{def:local_t} and \ref{def:mean_t}.
We note that when $X$ is a compact space which admits a universal covering, there is a standard way of taking $\alpha$ (which we call a \textit{standard representative}; see Definition \ref{def:standard} and Lemma \ref{lem:conti_theta}), and we can define translation numbers $\trot_{x}(\hg)$ and $\trot_{\mu}(\hg)$ as $a$-dependent notions;
see Remark \ref{rem:a-dep_trans_number}.

Let $\G = \Homeo(X,a)$ be the group of $a$-preserving homeomorphisms of $X$.
%Note that the bundle automorphism group $\hG$ surjects onto $\G$, and its kernel is isomorphic to $\A$.
In \cite{MR2854098}, Gal and K\k{e}dra introduced and studied a group two-cocycle $\mathfrak{G}_{x, \alpha}$ on $\G$ with coefficients in $\A$.
As the Poincar\'{e} translation number relates to the Euler class of $\Homeo_+(S^1)$, the local translation number and the mean translation number relate to the cohomology class $[\mathfrak{G}_{x, \alpha}]$ in $\HHH_{\grp}^2(\G;\A)$.
In particular, we obtain the following vanishing result on the class $[\mathfrak{G}_{x, \alpha}]$.
\begin{theorem}\label{thm_appendix}
  Let $K$ be a group and $\phi \colon K \to \G$ a homomorphism.
  If either the local translation number $\trot_{x, \alpha}$ or the mean translation number $\trot_{\mu, \alpha}$ is defined on the preimage $p^{-1}(\phi(K))$ of $\phi(K)$ and is a homomorphism on it, then the cohomology class $\phi^*[\mathfrak{G}_{x, \alpha}]$ is equal to zero in $\HHH_{\grp}^2(K;\RR)$.
\end{theorem}

By using Theorem \ref{thm_appendix}, we obtain the following, which was originally shown in Gal--K\k{e}dra \cite{1105.0825}.
\begin{theorem}[{\cite[Theorem 1.6]{1105.0825}}]\label{thm:app_meas_pres}
  Let $X$ be a compact space and $\mu$ a Borel probability measure on $X$.
  Assume that the class $a \in \HHH^1(X;\A)$ admits a standard representative.
  %Let $\Gm$ be the group of $\mu$-preserving homeomorphisms in $\G$.
  If a homomorphism $\phi \colon K \to \G$ factors through the group $\Gm$ of $\mu$-preserving homeomorphisms, then the class $\phi^*[\mathfrak{G}_{x, \alpha}]$ is equal to zero in $\HHH_{\grp}^2(K;\RR)$.
  %In particular, it is trivial on amenable subgroups of $\G$.
\end{theorem}
Under the settings in Theorem \ref{thm_appendix} (or in Theorem \ref{thm:app_meas_pres}), we further assume that the cohomology class $a$ is in $\HHH^2(X;\ZZ)$.
Then we can define an integer-valued cocycle $\mathfrak{G}_{x, \alpha, \ZZ}$ which satisfies $[\mathfrak{G}_{x, \alpha}] = [\mathfrak{G}_{x, \alpha, \ZZ}]$ in $\HHH_{\grp}^2(\G;\RR)$.
However, the integral cohomology class $\phi^*[\mathfrak{G}_{x, \alpha, \ZZ}] \in \HHH_{\grp}^2(K;\ZZ)$ does not necessarily equal to zero (see Remark \ref{rem:integral_not_vanish}).

%Note that, even if $\A = \ZZ$ in the assumption Theorem \ref{thm:app_meas_pres}, the class $\phi^*[\mathfrak{G}_{x, \alpha}]$ is not necessarily equal to zero in $\HHH^2(K;\ZZ)$.

We also provide several Seifert-fibered $3$-manifolds that the cohomology class $[\mathfrak{G}_{x, \alpha}]$ is non-zero (Theorem \ref{thm:app_seifert}).

%From here to the end of this section, we assume that the class $a$ admits a standard representative.
%In section \ref{sec:app}, we use the mean translation number and its relation to the class $[\mathfrak{G}_{x, \alpha}]$ to obtain a vanishing result (Theorem \ref{thm:app_meas_pres}) of the class $[\mathfrak{G}_{x, \alpha}]$, which was originally shown in Gal-K\k{e}dra \cite{1105.0825}.
%In Appendix \ref{sec:app}, we consider a group two-cocycle $\mathfrak{G}_{x, \alpha}$ on $\G$ defined in \cite{MR2854098}.
In Section \ref{sec:appl2}, %and \ref{sec:appl1}
we use the translation numbers to study distortion in the group $\hG$ of bundle automorphisms (see Subsection \ref{subsec:distortion} for the definition of distortion).
%The first one is about distortion (see Subsection \ref{subsec:distortion} for the definition of distortion).
%If we take a standard representative $\alpha$ of $a$ on a compact space $X$, the mean translation number exists and the local translation number exists $\mu$-a.e.; see Remark \ref{rem:existence_trans_numbers}.
The translation numbers
%for this cocycle $\alpha$
detect undistortion elements in $\hG$ as follows.

\begin{theorem}\label{thm:distortion}
  Let $X$ be a compact space and $\alpha$ a standard representative.
  If either a local translation number $\trot_{x, \alpha}(\hg)$ or a mean translation number $\trot_{\mu, \alpha}(\hg)$ is non-zero, then $\hg$ is undistorted in $\hG$.
\end{theorem}

In \cite{MR2854098}, Gal and K\k{e}dra studied undistortion elements in the group $\G = \Homeo(X, a)$.
%of $a$-preserving homeomorphisms of $X$.
Theorem \ref{thm:distortion} can be seen as a $\hG$-analogue of their results \cite[Theorems 1.1 and 1.6]{MR2854098}.

\section{Translation numbers}\label{sec:local_trans_number}

%\subsection{Definitions of the translation numbers}

In this section, we give the definitions of the local translation number $\trot_{x, \alpha}$ and the mean translation number $\trot_{\mu, \alpha}$.

Recall that the correspondence between a flat bundle and its holonomy homomorphism gives a bijection from the set of isomorphism classes of principal $\A$-bundles over $X$ to $\Hom(\pi_1(X),\A) \cong \HHH^1(X;\A)$.
Let $\pi \colon \hX \to X$ be a principal $\A$-bundle corresponding to the class $a \in \HHH^1(X;\A)$ and $\hG$ the group of bundle automorphisms of $\hX$.
Since $\hX \to X$ is a principal $\A$-bundle, the group $\A$ acts on $\hX$ and this defines a bundle automorphism $T_r \colon \hX \to \hX$ for any $r \in \A$.
Note that the bundle automorphism group $\hG$ is given as
\begin{align}\label{hG_def}
  \hG = \{ \hg \in \Homeo(\hX) \colon \hg \circ T_r = T_r \circ \hg \ \text{ for any } r \in \A \}.
\end{align}

For $a \in \HHH^1(X;\A)$, let $\alpha$ be a singular cocycle with coefficients in $\RR$ which represents $a$.
Here, when $\A = \ZZ$, we regard the class $a \in \HHH^1(X;\ZZ)$ as an element of $\HHH^1(X;\RR)$ under the change of coefficients homomorphism $\HHH^1(X;\ZZ) \to \HHH^1(X;\RR)$.
%(for a cocycle with coefficients in $\ZZ$, we will use the symbol $\alpha_{\ZZ}$).
Since $\pi^* a = 0$, there exists a singular zero-cochain $\theta$ with coefficients in $\RR$ on $\hX$ such that $d \theta = \pi^* \alpha$, where $d$ denotes the coboundary operator of the singular cochain complex.
We regard $\theta$ as a map from $\hX$ to $\RR$.

\begin{notation}
  For a singular chain $c$ and singular cochain $u$ on $X$, let $\int_{c}u$ denote their pairing.
\end{notation}

\begin{lemma}\label{lem:theta}
  There exists a cochain $\theta$ satisfying $d \theta = \pi^* \alpha$ and
  \begin{align}\label{eq:fiber_direction}
    \theta(T_r(\hx)) = \theta(\hx) + r
  \end{align}
  for any $\hx \in \hX$ and $r \in \A$.
\end{lemma}

\begin{proof}
  Let $\theta'$ be a cochain satisfying $d \theta' = \pi^* \alpha$.
  Let us fix a point $y$ in $X$ and $\hy$ a point in the fiber on $y$.
  First we consider the case where $\hy$ and $T_r(\hy)$ belong to the same path-connected component.
  We take a path $\gamma \colon [0,1] \to \hX$ such that $\gamma(0) = \hy$ and $\gamma(1) = T_r(\hy)$.
  Then, we have
  \begin{align*}
    \theta'(T_r(\hy)) - \theta'(\hy) = \int_{\gamma} d \theta' = \int_{\gamma} \pi^*\alpha = \int_{\pi(\gamma)} \alpha. %= \langle a, \pi(\gamma) \rangle.
  \end{align*}
  Since $\alpha$ represents the class $a$ and it corresponds to the holonomy homomorphism of $\hX \to X$, the value $\int_{\pi(\gamma)} \alpha$ is equal to the difference of $\gamma(1)$ and $\gamma(0)$, that is, equal to $r$.
  For any locally constant function $C \colon \hX \to \RR$, the sum $\theta = \theta' + C$ also satisfies $d \theta = \pi^* \alpha$.
  %By using a locally constant function $C$,
  %Hence we can modify the cochain $\theta'$ to satisfy the condition
  With a suitable choice of $C$, the cochain $\theta = \theta' + C$ satisfies
  \[
    \theta(T_r(\hy)) = \theta(\hy) + r
  \]
  for any $r \in \A$ and any $\hy$ in the fiber on $y$.
  This $\theta$ is a desired cochain.
  Indeed, let $x$ be an arbitrary point in $X$ and $\hx \in \hX$ its lift.
  For a path $\sigma$ from $x$ to $y$, there exists a unique lift $\tilde{\sigma}$ starting at $\hx$.
  Since the other endpoint is written as $T_s(\hy)$ for some $s \in \A$, we have
  \begin{align*}
    \theta(T_r(\hx)) - \theta(\hx) &= \theta(T_{r+s}(\hy)) - \int_{T_r\tilde{\sigma}} d \theta - \left( \theta(T_s(\hy)) - \int_{\tilde{\sigma}} d \theta \right)\\
    & = r - \int_{\sigma} \alpha + \int_{\sigma} \alpha = r.
  \end{align*}
\end{proof}

\begin{definition}
  Let $x$ be a point in $X$ and $\theta$ a singular zero-cochain on $\hX$ as in Lemma \ref{lem:theta}.
  The function $\rho_{x, \alpha} \colon \hG \to \A$ is defined by
  \[
    \rho_{x, \alpha}(\hg) = \theta(\hg(\hx)) - \theta(\hx),
  \]
  where $\hx$ is a point in the fiber on $x$.
\end{definition}

Note that, by Lemma \ref{lem:theta}, the function $\rho_{x, \alpha}$ is independent of the choice of $\theta$ and the lift $\hx$ of $x$.

\begin{definition}[Local translation number]\label{def:local_t}
  For an element $\hg \in \hG$, its \textit{local translation number} $\trot_{x, \alpha}(\hg)$ is defined by
  \[
    \trot_{x, \alpha}(\hg) = \lim_{n \to \infty} \frac{\rho_{x, \alpha}(\hg^n)}{n}
  \]
  if this limit exists.
\end{definition}

\begin{remark}\label{rem:differ_from_GK}
  In \cite[Proposition 3.10]{MR2854098}, Gal and K\k{e}dra considered the value
  \begin{align}\label{def:GKlocal}
    \lim_{n \to \infty} \frac{\theta(\hg^n(\hx))}{n}
  \end{align}
  for an arbitrary zero-cochain $\theta$ satisfying $d\theta = \pi^*\alpha$ (i.e., $\theta$ does not necessarily satisfy condition (\ref{eq:fiber_direction})) when $\A = \ZZ$.
  If $\A = \RR$, limit (\ref{def:GKlocal}) depends on the choice of $\theta$.
\end{remark}

If the function $\rho_{x, \alpha}$ is a quasimorphism on the cyclic subgroup $\langle \hg \rangle \subset \hG$, that is, the set $\{ \rho_{x, \alpha}(\hg^{m+n}) - \rho_{x, \alpha}(\hg^m) - \rho_{x, \alpha}(\hg^n) \mid m, n \in \ZZ \}$ is bounded,
then the limit exists (see \cite[Lemma 2.21]{Cal} for example). %Subsection \ref{subsec:translation_galkedra} for quasimorphism).
In particular, if $\hg$ preserves the fiber on $x$ (as a set), the local translation number $\trot_{x, \alpha}(\hg)$ is defined.

Another condition that guarantees the existence of the limit comes from ergodic theory.
Let $g \colon X \to X$ be a homeomorphism covered by $\hg$ and $\mu$ a $g$-invariant Borel probability measure on $X$.
If the function $\rho_{x, \alpha}(\hg)$ is in $L^1(\mu)$, the limit exists $\mu$-a.e. and $\trot_{x, \alpha}(\hg)$ is also in $L^1(\mu)$ by Birkhoff's ergodic theorem.

\begin{definition}[Mean translation number]\label{def:mean_t}
  For an element $\hg \in \hG$, its \textit{mean translation number} $\trot_{\mu, \alpha}(\hg)$ is defined by
  \[
    \trot_{\mu, \alpha}(\hg) = \int_{X} \rho_{x, \alpha}(\hg) \, d\mu(x)
  \]
  if this integral exists.
\end{definition}

\begin{remark}\label{rem:mean_t}
  \begin{enumerate}[$(1)$]
    \item Birkhoff's ergodic theorem also implies that $\trot_{\mu, \alpha}(\hg) = \displaystyle\int_{X}\trot_{x, \alpha}(\hg) \, d\mu(x)$.
    Moreover, if $g$ is $\mu$-ergodic, the local translation number $\trot_{x, \alpha}(\hg)$ is constant $\mu$-a.e. and equal to the mean translation number $\trot_{\mu, \alpha}(\hg)$ $\mu$-a.e.
    \item For any homeomorphism $g$ on a compact metric space $X$, there always exists a $g$-invariant Borel probability measure $\mu$ on $X$ by the Krylov--Bogolioubov theorem (see \cite[Corollary 6.9.1]{MR648108} for example).
  \end{enumerate}
\end{remark}

%If $X = S^1$ and $a \in \HHH^1(S^1;\ZZ)$ is the generator, then the local translation number and the mean translation number are equal to Poincar\'{e}'s translation number.

\begin{proposition}\label{prop:trans_number_hom}
  Let $g$ and $h$ be elements of $\G$ covered by $\hg$ and $\hh$ in $\hG$, respectively.
  \begin{enumerate}[$(1)$]
    \item If $g$ and $h$ preserve a point $x$ in $X$, then
    \[
      \trot_{x, \alpha}(\hg \hh) = \trot_{x, \alpha}(\hg) + \trot_{x, \alpha}(\hh).
    \]
    \item If $g$ and $h$ preserve a Borel probability measure $\mu$ on $X$, then
    \[
      \trot_{\mu, \alpha}(\hg \hh) = \trot_{\mu, \alpha}(\hg) + \trot_{\mu, \alpha}(\hh)
    \]
    if each term exists.
  \end{enumerate}
\end{proposition}

\begin{proof}
  Since
  \[
    \rho_{x, \alpha}(\hg \hh) = \theta(\hg\hh(\hx)) - \theta(\hh(\hx)) + \theta(\hh(\hx)) - \theta(\hx) = \rho_{h(x), \alpha}(\hg) + \rho_{x, \alpha}(\hh)
  \]
  for any $\hg$ and $\hh$ in $\hG$, the proposition follows.
\end{proof}

\begin{proposition}\label{prop:periodic}
  Let $\A$ be the group $\ZZ$ of integers and $g$ an element of $\G$ covered by $\hg \in \hG$.
  If $g$ has a periodic point $x$ of period $q$, then the following hold:
  \begin{enumerate}[$(1)$]
    \item The local translation number $\trot_{x, \alpha}(\hg)$ is of the form $n/q$ for some $n \in \ZZ$.
    \item Assume that $\rho_{x, \alpha}(\hg)$ is measurable.
    Then, there exists a $g$-invariant Borel probability measure $\mu$ such that the mean translation number $\trot_{\mu, \alpha}$ is equal to $n/q$ for some $n \in \ZZ$.
  \end{enumerate}
\end{proposition}

\begin{proof}
  Take a lift $\hx \in \hX$ of $x$.
  By the assumption of $x$, there exists an integer $n \in \ZZ$ such that $\hg^q(\hx) = T_n (\hx)$.
  \begin{enumerate}[$(1)$]
    \item Notice that the local translation number for a periodic point $x$ always exists since $\rho_{x, \alpha}$ is a quasimorphism on the cyclic subgroup $\langle \hg \rangle$.
    Since $\theta(T_n(\hx)) = \theta(\hx) + n$, we have $\rho_{x, \alpha}(\hg^{qk}) = nk$ for any $k \in \ZZ$.
    Hence we have
    \begin{align*}
      \trot_{x, \alpha}(\hg) = \lim_{k \to \infty} \frac{\rho_{x, \alpha}(\hg^k)}{k} = \lim_{k \to \infty} \frac{\rho_{x, \alpha}(\hg^{qk})}{qk} = \frac{1}{q} \lim_{k \to \infty} \frac{nk}{k} = \frac{n}{q}.
    \end{align*}
    \item For $y \in X$, let $\delta_y$ denote the Dirac measure on $y$, and we define a $g$-invariant Borel probability measure $\mu$ by
    \[
      \mu = \frac{1}{q} \sum_{i = 0}^{q-1} \delta_{g^{i}(x)}.
    \]
    By (1), we have $\trot_{g^i(x), \alpha}(\hg) = n/q$ for any $0 \leq i \leq q-1$.
    Hence we obtain
    \[
      \trot_{\mu, \alpha}(\hg) = \int_{X} \trot_{x, \alpha}(\hg) \, d\mu(x) = \frac{n}{q}.
    \]
  \end{enumerate}
\end{proof}

\begin{remark}
  Even if the translation numbers $\trot_{x, \alpha}(\hg)$ and $\trot_{\mu, \alpha}(\hg)$ are equal to zero, the homeomorphism $g$ may not have a periodic point.
  Let $X = T^2$ be the two-dimensional torus and $a$ the cohomology class corresponding to $(1, 0) \in \ZZ \times \ZZ = \HHH^1(X;\ZZ)$.
  Then the space $\hX$ is homeomorphic to $\RR \times S^1$.
  For an irrational number $r$, we define $g_r \colon X \to X$ and its lift $\hg_r$ by
  \[
    g_r (z,w) = (z, w e^{2\pi ir})
  \]
  and
  \[
    \hg_r (x,w) = (x, w e^{2\pi ir}),
  \]
  respectively.
  Then, for $\alpha = dz$, the translation numbers are equal to zero although $g_r$ has no periodic points.
\end{remark}

Under the setting in Proposition \ref{prop:periodic}, the local translation number $\trot_{x, \alpha}(\hg)$ and the mean translation number $\trot_{\mu, \alpha}(\hg)$ are independent of the choice of $\alpha$.
This is generalized as follows:
\begin{proposition}\label{prop:alpha-independence}
  Let $\alpha_1$ and $\alpha_2$ be singular one-cocycles representing the class $a \in \HHH^1(X;\A)$ and $\beta$ a singular zero-cochain on $X$ satisfying $\alpha_1 - \alpha_2 = d\beta$.
  We regard $\beta$ as a function $\beta \colon X \to \A$.
  Let $g$ be an element of $\G$ covered by $\hg \in \hG$.
  \begin{enumerate}[$(1)$]
    \item Assume that $\trot_{x, \alpha_1}(\hg)$ and $\trot_{x, \alpha_2}(\hg)$ are defined.
    If $\beta$ is a bounded function on the orbit $\{ g^i(x) \}_{i \in \NN}$, the equality
    \[
      \trot_{x, \alpha_1}(\hg) = \trot_{x, \alpha_2}(\hg)
    \]
    holds.
    \item Let $\mu$ be a $g$-invariant Borel probability measure on $X$ and assume that $\trot_{\mu, \alpha_1}(\hg)$ and $\trot_{\mu, \alpha_2}(\hg)$ are defined.
    If $\displaystyle \int_{X} \bigl( \beta(g(x)) - \beta(x) \bigr) d\mu(x) = 0$, the equality
    \[
      \trot_{\mu, \alpha_1}(\hg) = \trot_{\mu, \alpha_2}(\hg)
    \]
    holds.
  \end{enumerate}
\end{proposition}

\begin{remark}
  If $\beta \colon X \to \A$ is measurable, the condition $\displaystyle \int_{X} \bigl( \beta(g(x)) - \beta(x) \bigr) d\mu(x) = 0$ holds.
\end{remark}

\begin{proof}[Proof of Proposition $\ref{prop:alpha-independence}$]
  Let $\theta_j$ be a zero-cochain on $\hX$ for $\alpha_j$ as in Lemma \ref{lem:theta}.
  Since $\alpha_1 - \alpha_2 = d\beta$, we have $d(\theta_1 - \theta_2) = d(\pi^*\beta)$.
  Hence there exists a locally constant function $C \colon \hX \to \RR$ such that $\theta_1 - \theta_2 - \pi^* \beta = C$.
  This $C$ is in fact globally constant since the equality
  \[
    C(T_r(\hy)) = \theta_1(T_r(\hy)) - \theta_2(T_r(\hy)) - \pi^*\beta(T_r(\hy)) = \theta_1(\hy) - \theta_2(\hy) - \pi^*\beta(\hy) = C(\hy)
  \]
  holds for any $\hy \in \hX$ and $r \in \A$.
  Hence we obtain
  \[
    \rho_{x, \alpha_1}(\hg) - \rho_{x, \alpha_2}(\hg) = \beta(g(x)) - \beta(x),
  \]
  and this immediately implies the proposition.
\end{proof}

\begin{definition}\label{def:standard}
  Let $X$ be a topological space and $a \in \HHH^1(X;\A)$.
  A singular one-cocycle $\alpha$ representing $a$ is called \textit{standard} if there exists a zero-cochain $\theta \colon \hX \to \RR$ such that $\theta$ is continuous, $\pi^* \alpha = d\theta$, and
  \[
    \theta(T_r(\hx)) = \theta(\hx) + r
  \]
  for any $\hx \in \hX$ and $r \in \A$.
  Let $S_a$ be the set of standard cocycles of $a$.
\end{definition}

\begin{remark}\label{rem:existence_trans_numbers}
  For a standard representative $\alpha$ of $a$, the function $\rho_{x, \alpha}(\hg)$ is continuous with respect to $x$ for any $\hg \in \hG$.
  Hence, if we assume that $X$ is compact, the function $\rho_{x, \alpha}(\hg)$ is in $L^1(\mu)$ for any Borel probability measure on $X$.
  In particular, for $p(\hg)$-invariant Borel probability measure $\mu$, the mean translation number $\trot_{\mu, \alpha}(\hg)$ always exists and the local translation number $\trot_{x, \alpha}(\hg)$ exists $\mu$-a.e. %by Birkhoff's ergodic theorem.
\end{remark}

If the space $X$ is a smooth manifold, we can take $\alpha$ as a differential one-form and $\theta$ as a differential zero-form, that is, a smooth function on $\hX$.
Hence any cohomology class on a smooth manifold admits a standard representative.
%$S_a$ is non-empty for any smooth manifold.
More generally, the following lemma holds,
%The following lemma guarantees the existence of a standard cocycle in a broader class of spaces,
which was essentially proved in \cite[Lemma 2.1]{MR2854098}.

\begin{lemma}\label{lem:conti_theta}
  Let $X$ be a paracompact space which admits a universal covering.
  Then, for any $a \in \HHH^1(X;\A)$, there exists a standard representative.
  In particular, $S_a$ is non-empty.
\end{lemma}

\begin{proof}
  By abuse of notation, let $a \colon \pi_1(X) \to \A$ denote the holonomy homomorphism.
  When $\A = \ZZ$, we consider the homomorphism $a \colon \pi_1(X) \to \ZZ$ as an real-valued homomorphism by composition with the inclusion $\ZZ \to \RR$.
  Let $\RR \to \widetilde{X}\times_{\pi_1(X)} \RR \xrightarrow{\pi} X$ be the $\RR$-bundle associated to the holonomy homomorphism $a$, where $\RR$ is endowed with the usual topology.
  Since the fiber is contractible and $X$ is paracompact, there exists a continuous section $\sigma \colon X \to \widetilde{X}\times_{\pi_1(X)} \RR$.
  We define a map $\theta \colon \widetilde{X}\times_{\pi_1(X)} \RR \to \RR$ by
  \[
    [\widetilde{x}, r] = T_{\theta([\widetilde{x}, r])}(\sigma\pi([\widetilde{x}, r])).
  \]
  Then, the function $\theta$ satisfies $\theta(T_s([\widetilde{x}, r])) = \theta([\widetilde{x}, r]) + s$.
  Moreover, since the section $\sigma$ is continuous, so is $\theta$.
  The $\A$-bundle $\hX = \widetilde{X}\times_{\pi_1(X)} \A$ is a subset of $\widetilde{X}\times_{\pi_1(X)} \RR$ with a finer topology.
  Hence the map $\theta$ remains continuous on $\hX$.
  Moreover, there exists a singular one-cocycle $\alpha$ such that $d \theta = \pi^* \alpha$.
  Indeed, for paths $\hgg$ and $\hgg'$ on $\hX$ satisfying $\pi_*\hgg = \pi_*\hgg'$, there exists a real number $r \in \RR$ such that $\hgg = T_r \circ \hgg'$ since the fiber of $\hX$ has the discrete topology.
  Hence we have
  \begin{align*}
    \int_{\hgg} d \theta &= \theta (\hgg(1)) - \theta(\hgg(0)) = \theta(T_r(\hgg'(1))) - \theta(T_r(\hgg'(0)))\\
    &= \theta(\hgg'(1)) + r - \bigl(\theta(\hgg'(0)) + r \bigr) = \int_{\hgg'} d \theta.
  \end{align*}

  Finally, we show that the cocycle $\alpha$ represents the class $a$.
  It suffices to show that, for any loop $\gg$ in $X$, the value $\int_{\gamma} \alpha$ is equal to the holonomy of $\gg$.
  Let $\hgg \colon [0,1] \to \hX$ be a lift of $\gg$.
  Then we have
  \begin{align*}
    \int_{\gg} \alpha = \int_{\pi_* \hgg} \alpha = \int_{\hgg} \pi^* \alpha = \int_{\hgg} d \theta = \theta(\hgg(1)) - \theta(\hgg(0)),
  \end{align*}
  %Since $\theta$ satisfies $\theta(T_r(\hx)) = \theta(\hx) + r$, the value $\theta(\hgg(1)) - \theta(\hgg(0))$ is the holonomy of $\gg$.
  and hence $\alpha$ represents the class $a$.
\end{proof}

\begin{proposition}\label{prop:S_a-indep}
  Let $X$ be a compact space and $a \in \HHH^1(X;\A)$.
  Assume that $S_a$ is non-empty.
  Then, the local translation number $\trot_{x, \alpha}$ and the mean translation number $\trot_{\mu, \alpha}$ are independent of the choice of $\alpha \in S_a$.
\end{proposition}

\begin{proof}
  For elements $\alpha_1$ and $\alpha_2$ of $S_a$, we take zero-cochains $\theta_1$ and $\theta_2$, respectively, as in the definition of $S_a$.
  Then there exists a zero-cochain $\beta \colon \hX \to \RR$ and a constant function $C$ such that
  \[
    \theta_1 - \theta_2 = \pi^*\beta + C
  \]
  by the same argument in the proof of Proposition \ref{prop:alpha-independence}.
  Since $\theta_1$ and $\theta_2$ are continuous, so is $\beta \colon X \to \RR$.
  Hence, Proposition \ref{prop:alpha-independence} and the compactness of $X$ imply the proposition.
\end{proof}

\begin{remark}\label{rem:a-dep_trans_number}
  Let $X$ be a compact space and $a \in \HHH^1(X;\A)$ a cohomology class which admits a standard representative.
  For a point $x$ and a Borel probability measure $\mu$, we define $\trot_{x}$ and $\trot_{\mu}$ by
  \[
    \trot_{x} = \trot_{x, \alpha}  \ \ \text{ and }  \ \  \trot_{\mu} = \trot_{\mu,\alpha}
  \]
  for some $\alpha \in S_a$.
  These translation numbers are well-defined and $a$-dependent notions by Proposition \ref{prop:S_a-indep}, although they will not appear in the rest of the paper.
\end{remark}

\subsection{Relation to the homological translation vector}\label{subsec:trans_vector}
Let us recall the definition of the homological translation number (see \cite{MR1325916} and \cite{MR1444450} for details).
Let $X$ be a closed manifold and $\{ g_t \}_{0 \leq t \leq 1}$ an isotopy from the identity $\id_X = g_0$ to a homeomorphism $g = g_1$.
Let $\tg$ denote the isotopy class of $\{ g_t \}$ relative to the fixed endpoints.
For a point $x$ in $X$, the homological translation vector $h_{x, \tg}$ of $x$ is defined as an element of
\[
  \HHH_1(X;\RR) \cong \Hom([X,S^1],\RR),
\]
where $[X,S^1]$ is the homotopy set of continuous maps $X \to S^1$.
For a continuous map $\varphi \colon X \to S^1 = \RR/\ZZ$ and a path $\gamma \colon [0,1] \to X$, let $\varphi_{\gamma} \colon [0,1] \to \RR$ be a continuous lift of the map $t \mapsto \varphi(\gamma(t))$, and we set
\[
  \Delta_{\varphi}(\gamma) = \varphi_{\gamma}(1) - \varphi_{\gamma}(0).
\]
Then, the homological translation vector $h_{x, \tg} \colon [X,S^1] \to \RR$ is defined by
\[
  h_{x, \tg} ([\varphi]) = \lim_{n \to \infty} \frac{1}{n} \Delta_{\varphi}(\{ g_t(x) \}\ast \{ g_t \circ g(x) \} \ast \cdots \ast \{ g_t \circ g^{n-1}(x) \})
\]
if the limit exists.
Here $\ast$ denotes the concatenation of the paths.
For a $g$-invariant Borel probability measure $\mu$, the mean translation number $h_{\mu, \tg} \in H_1(X;\RR)$ is defined by
\[
  h_{\mu, \tg}([\varphi]) = \int_{X} \Delta_{\varphi}(\{ g_t(x) \}) \, d\mu(x).
\]

For a continuous map $\varphi \colon X \to S^1$, let $\alpha_{\varphi}$ be the one-cocycle on $X$ induced by $\Delta_{\varphi}$.
Note that this induces the isomorphism $[X, S^1] \cong \mathrm{Im}(\HHH^1(X;\ZZ) \to \HHH^1(X;\RR) )$.
We set $a = [\alpha_{\varphi}] \in \HHH^1(X;\RR)$.
For an isotopy $\{ g_t \}_{0 \leq t \leq 1}$, let $\{ \hg_t \colon \hX \to \hX \}_{0 \leq t \leq 1}$ be the lift from $\id_{\hX} = \hg_0$, and we set $\hg = \hg_1$.
\begin{proposition}
  The equalities
  \[
    h_{x, \tg}([\varphi]) = \trot_{x, \alpha_{\varphi}}(\hg) \ \ \text{ and } \ \ h_{\mu, \tg}([\varphi]) = \trot_{\mu, \alpha_{\varphi}}(\hg)
  \]
  hold.
\end{proposition}

\begin{proof}
  We take a zero-cochain $\theta_{\varphi}$ for $\alpha_{\varphi}$ by Lemma \ref{lem:theta}.
  Since $d\theta_{\varphi} = \pi^* \alpha_{\varphi}$, we have
  \begin{align*}
    \Delta_{\varphi}(\{ g_t(x) \}) = \int_{\{ g_t(x) \}} \alpha_{\varphi} = \int_{\{ \hg_t(\hx) \}} d\theta_{\varphi} = \theta_{\varphi}(\hg(\hx)) - \theta_{\varphi}(\hx).
  \end{align*}
  This implies the proposition.
\end{proof}

%\subsection{{\maru Geometric meaning of the translation numbers\relax}}

\section{On the Gal--K\k{e}dra cocycle}\label{sec:app}

The goal of this section is to show Theorems \ref{thm_appendix} and \ref{thm:app_meas_pres}.

\subsection{Preliminary on group cohomology}\label{subsec:group_coh}
In this subsection, we review group cohomology (see \cite{brown82} for details).
For a group $G$ and a trivial $G$-module $\Z$, a \textit{group $p$-cochain of $G$ with coefficients in $\Z$} is a function from the $n$-fold product $G^n$ to $\Z$.
Let $C_{\grp}^n(G;\Z)$ denote the module of group $n$-cochains.
The coboundary map $\delta \colon C_{\grp}^n(G;\Z) \to C_{\grp}^{n+1}(G;\Z)$ is defined by
\begin{align*}
  \delta c (g_1, \dots, g_{n+1}) = & c(g_2, \dots, g_{n+1}) + \sum_{i = 1}^{n}(-1)^i c(g_1, \dots, g_i g_{i+1}, \dots, g_{n+1}) \\
  & +(-1)^{n+1}c(g_1, \dots, g_n)
\end{align*}
for $n > 0$ and $\delta = 0$ for $n = 0$.
The cohomology of the cochain complex $(C_{\grp}^*(G;\Z),\delta)$ is  denoted by $H_{\grp}^*(G;\Z)$ and called the {\it group cohomology of $G$}.
By definition, we have $H_{\grp}^0(G;\Z) \cong \Z$ and $H_{\grp}^1(G;\Z) \cong \Hom(G,\Z)$, where $\Hom(G,\Z)$ denotes the module of $\Z$-valued homomorphisms on $G$.

The second cohomology group $H_{\grp}^2(G;\Z)$ is closely related to the central $\Z$-extensions of $G$.
An exact sequence $1 \to \Z \xrightarrow{i} E \xrightarrow{p} G \to 1$ is called a \textit{central $\Z$-extension of $G$} if the image $i(\Z)$ is contained in the center of $E$.
%To show Theorem \ref{thm_appendix}, we use the following facts.

\begin{theorem}[{\cite[(3.12) Theorem]{brown82}}]\label{thm:app_cent_ext_coh}
  The second group cohomology $\HHH_{\grp}^2(G;\Z)$ with coefficients in $\Z$ is bijective to the set of equivalence classes of central $\Z$-extensions of $G$;
  \begin{align}\label{sec_coh_cent}
    \HHH_{\grp}^2(G;\Z) \cong \{ \text{central $\Z$-extensions of $G$}\} / \{ \text{splitting extensions} \}.
  \end{align}
\end{theorem}

For a central $\Z$-extension $E$, the corresponding cohomology class $e(E)$ under bijection (\ref{sec_coh_cent}) is called the \textit{Euler class of the central extension $E$}.
One of the cocycles representing the Euler class $e(E)$ is given as the following proposition:
\begin{proposition}[{\cite[Proposition 1]{Moriyoshi16}}]\label{prop:app_euler_cocycle_conn}
  Let $\Z$ and $\W$ be trivial $G$-modules, $0 \to \Z \xrightarrow{i} E \xrightarrow{p} G \to 1$ be a central extension, and $f \colon \Z \to \W$ a homomorphism.
  Let $F \colon E \to \W$ be a function whose restriction $F|_{\Z}$ is equal to $f$, where we identify $i(\Z)$ with $\Z$.
  Then, there exists a cocycle $c \in C_{\grp}^2(G;\W)$ such that the pullback $p^*c$ is equal to $-\delta F$.
  Moreover, the cocycle $c$ represents the class $f_* e(E) \in \HHH_{\grp}^2(G;\W)$, where $f_* \colon \HHH_{\grp}^2(G;\Z) \to \HHH_{\grp}^2(G;\W)$ is the change of coefficients homomorphism.
\end{proposition}

\begin{remark}\label{rem:app_hom_ext_trivial}
  For a central extension $1 \to \Z \to E \to G \to 1$ and a homomorphism $f \colon \Z \to \W$, the class $f_*e(E)$ is trivial if and only if the homomorphism $f \colon \Z \to \W$ extends to a homomorphism $F \colon E \to \W$ by Proposition \ref{prop:app_euler_cocycle_conn}.
  In particular, the Euler class $e(E)$ is trivial if and only if the identity homomorphism $\id_{\Z}$ extends to a homomorphism $F \colon E \to \Z$.
\end{remark}

\subsection{Gal--K\k{e}dra cocycle}
For $a \in \HHH^1(X;\A)$, let $\G = \Homeo(X, a)$ be the group of $a$-preserving homeomorphisms of $X$.
Gal and K\k{e}dra defined in \cite{MR2854098} a group two-cocycle $\mathfrak{G}_{x, \alpha}$ on $\G$, which we call the \emph{Gal--K\k{e}dra cocycle}.
For a point $x \in X$ and a representative $\alpha$ of $a \in \HHH^1(X;\A)$, the cocycle %with coefficients in $\A$
is defined by
\[
  \mathfrak{G}_{x, \alpha}(g,h) = \int_x^{h(x)} g^* \alpha - \alpha
\]
for $g, h \in \G$.
Here, the symbol $\int_x^{h(x)} \sigma$ denotes the pairing of a singular one-cochain $\sigma$ and a path from $x$ to $h(x)$.
Since $g^* \alpha - \alpha$ is a coboundary, $\mathfrak{G}_{x, \alpha}(g,h)$ is independent of the choice of paths from $x$ to $h(x)$.
This group two-cochain $\mathfrak{G}_{x, \alpha}$ is a cocycle, and its cohomology class is independent of the choice of $x$ and $\alpha$ (see \cite{MR2854098} for details; see also \cite{ismagilov_losik_michor06}).
%For the terminology of group cohomology theory, we refer to \cite{brown82}.

%In this appendix, we deduce the cohomological triviality of $\mathfrak{G}_{x, \alpha}$ from a property local translation number $\trot_{x, \alpha}$ as follows.
%Let $K$ be a group and $\phi \colon K \to \G$ be a homomorphism.
Note that the projection $\hG \to \Homeo(X)$ factors through $\G$.
Moreover, the induced map $p \colon \hG \to \G$ is surjective.
In particular, we have a central extension
\begin{align}\label{cent_ext_hG}
  0 \to \A \to \hG \xrightarrow{p} \G \to 1.
\end{align}

\begin{comment}
Let us consider the following commutative diagram
\[
\xymatrix{
0 \ar[r] & \A \ar[r] & \hG \ar[r] & \G \ar[r] & 1 \\
0 \ar[r] & \A \ar[r] \ar@{=}[u] & \phi^*\hG \ar[r] \ar[u]^-{\tphi} & K \ar[r] \ar[u]^-{\phi} & 1,
}
\]
where $\phi^*\hG$ is the pullback of $\hG$.
\begin{theorem}%\label{thm_appendix}
  If the local translation number $\trot_{x, \alpha}$ is defined on the preimage $p^{-1}(\phi(K))$ of $\phi(K)$ and is a homomorphism on it, then the cohomology class $\phi^*[\mathfrak{G}_{x, \alpha}]$ is equal to zero in $\HHH_{\grp}^2(K;\RR)$.
\end{theorem}
\end{comment}

In \cite{1105.0825}, the following is shown.

\begin{theorem}[{\cite[Theorem 5.1]{1105.0825}}]\label{thm:app_gal_kedra_euler}
  %The equalities $p^* \mathfrak{G}_{x, \alpha} = - \delta \rho_{x, \alpha}$ and $\rho_{x, \alpha}|_{\A} = \id_A$ hold.
  %In particular,
  The cocycle $\mathfrak{G}_{x, \alpha}$ represents the Euler class $e(\hG)$ of central extension (\ref{cent_ext_hG}).
\end{theorem}

To see the relation between the cocycle $\mathfrak{G}_{x, \alpha}$ and the local rotation number (or the function $\rho_{x, \alpha}$), we state a proof of Theorem \ref{thm:app_gal_kedra_euler} using $\rho_{x, \alpha}$.

\begin{proof}[Proof of Theorem \textup{\ref{thm:app_gal_kedra_euler}}]
  For any $r \in \A$, we have
  \[
    \rho_{x, \alpha}(T_r) = \theta(T_r(\hx)) - \theta(\hx) = \theta(\hx) + r - \theta(\hx) = r
  \]
  by Lemma \ref{lem:theta}.
  Hence the restriction $\rho_{x, \alpha}|_{\A}$ is equal to the identity homomorphism $\id_{\A}$.

  For any $\hg, \hh \in \hG$, we have
  \begin{align*}
    -\delta \rho_{x, \alpha} (\hg, \hh) &= \theta(\hg \hh(\hx)) - \theta(\hx) - \bigl(\theta(\hg(\hx))- \theta(\hx) + \theta(\hh(\hx)) - \theta(\hx)\bigr) \\
    &= \theta(\hg \hh (\hx)) - \theta(\hh (\hx)) - \bigl( \theta(\hg(\hx)) - \theta(\hx) \bigr).
  \end{align*}
  We set $p(\hg) = g$ and $p(\hh) = h$, then we have
  \begin{align*}
    p^*\mathfrak{G}_{x, \alpha}(\hg, \hh) = \mathfrak{G}_{x, \alpha}(g,h) = \int_{x}^{h(x)} g^* \alpha - \alpha.
  \end{align*}
  Let $\gamma \colon [0,1] \to X$ be a path satisfying $\gamma(0) = x$ and $\gamma(1) = h(x)$.
  We take the lift $\hgg \colon [0,1] \to \hX$ of $\gamma$ satisfying $\hgg(0) = \hx$.
  Then we obtain
  \begin{align*}
    \int_{x}^{h(x)} g^* \alpha - \alpha &= \int_{\gamma} g^* \alpha - \alpha = \int_{\pi(\hgg)} g^*\alpha - \alpha = \int_{\hgg} \pi^* g^* \alpha - \pi^* \alpha \\
    &= \int_{\hgg} \hg^* \pi^* \alpha - \pi^* \alpha = \int_{\hgg} \hg^* d \theta - d \theta = \int_{\hgg}d(\hg^* \theta - \theta) \\
    &= \int_{\partial \hgg} \hg^* \theta - \theta = \theta(\hg(\hgg(1))) - \theta(\hgg(1)) - \bigl( \theta(\hg(\hgg(0))) - \theta(\hgg(0)) \bigr) \\
    &= \theta(\hg(\hgg(1))) - \theta(\hgg(1)) - \bigl( \theta(\hg(\hx)) - \theta(\hx) \bigr).
  \end{align*}
  Since $\hgg(1)$ is in the fiber on $h(x)$, there exists a number $r \in \A$ such that $\hgg(1) = T_r(\hh(\hx))$.
  Hence we obtain
  \begin{align*}
    \theta(\hg(\hgg(1))) - \theta(\hgg(1)) &= \theta(\hg(T_r(\hh(\hx)))) - \theta(T_r(\hh(\hx))) \\
    &= \theta(T_r(\hg\hh(\hx))) - \theta(T_r(\hh(\hx)))\\
    &= \theta(\hg\hh(\hx)) + r - \bigl(\theta(\hh(\hx)) + r \bigr) = \theta(\hg\hh(\hx)) - \theta(\hh(\hx)).
  \end{align*}
  This implies $p^* \mathfrak{G}_{x, \alpha} = - \delta \rho_{x, \alpha}$.
  By Proposition \ref{prop:app_euler_cocycle_conn}, we obtain $[\mathfrak{G}_{x, \alpha}] = e(\hG)$.
\end{proof}

\begin{remark}
  \begin{enumerate}[$(1)$]
    \item In the proof of Theorem \ref{thm:app_gal_kedra_euler}, we showed that the equality $p^* \mathfrak{G}_{x, \alpha} = - \delta \rho_{x, \alpha}$ holds.
    Let $\G^b$ be the subgroup of $\G$ on which the cocycle $\mathfrak{G}_{x, \alpha}$ is bounded and $\hG^b$ the preimage of $\G^b$ with respect to the projection $\hG \to \G$.
    The equality $p^* \mathfrak{G}_{x, \alpha} = - \delta \rho_{x, \alpha}$ implies that the function $\rho_{x, \alpha}$ is a quasimorphism on $\hG^b$.
    Hence, the local translation number gives rise to a homogeneous quasimorphism
    \[
      \trot_{x, \alpha} \colon \hG^b \to \RR.
    \]
    (See \cite{Cal} for basics of quasimorphisms.)
    \item In general, the cocycle $\mathfrak{G}_{x, \alpha}$ is not bounded on $\G$.
    For example, let $T^n$ be a $n$-dimensional torus for $n \geq 2$.
    Let $i \colon S^1 \to T^n$ be the injection defined by $i(x) = (x, 0, \cdots, 0)$ and $\phi \colon \Homeo_+(S^1) \to \Homeo_0(T^n)$ the homomorphism defined by
    \[
      (\phi(f))(x_1, x_2, \cdots, x_n) = (f(x_1), x_2, \cdots, x_n),
    \]
    where $x \in S^1$, $f \in \Homeo_+(S^1)$ and $(x_1, \cdots, x_n) \in T^n$.
    If $a \in H^1(T^n;\ZZ)$ is a class satisfying $i^*a \neq 0$, the pullback $\phi^*[\mathfrak{G}_{x, \alpha}]$ is equal to the Euler class of $\Homeo_+(S^1)$ up to non-zero constant multiple.
    In particular, the class $[\mathfrak{G}_{x, \alpha}]$ is non-zero on $\Homeo_0(T^n)$.
    Because any non-zero cohomology class in $H_{\grp}^2(\Homeo_0(T^n);\ZZ)$ is unbounded \cite{MR}, the cocycle $\mathfrak{G}_{x, \alpha}$ is not bounded for any $x$ and $\alpha$.
  \end{enumerate}
\end{remark}

\begin{proof}[Proof of Theorem \textup{\ref{thm_appendix}}]
  Let us consider the following commutative diagram
  \[
  \xymatrix{
  0 \ar[r] & \A \ar[r] & \hG \ar[r] & \G \ar[r] & 1 \\
  0 \ar[r] & \A \ar[r] \ar@{=}[u] & \phi^*\hG \ar[r] \ar[u]^-{\tphi} & K \ar[r] \ar[u]^-{\phi} & 1,
  }
  \]
  where $\phi^*\hG$ is the pullback of $\hG$.
  By assumption, either the local translation number $\trot_{x, \alpha}$ or the mean translation number $\trot_{\mu, \alpha}$ defines a homomorphism $\phi^*\hG \to \RR$ whose restriction to $\A$ is equal to the inclusion map $\A \to \RR$.
  Hence the Euler class $e(\phi^*\hG)$ is trivial in $\HHH_{\grp}^2(K;\RR)$ by Remark \ref{rem:app_hom_ext_trivial}.
  On the other hand, the class $e(\phi^*\hG)$ is equal to the pullback $\phi^* (e(\hG))$.
  Hence the class $\phi^* (e(\hG))$ is equal to zero.
\end{proof}

\begin{comment}
Theorem \ref{thm_appendix}, together with Proposition \ref{prop:trans_number_hom}, implies the following triviality theorem by Gal--K\k{e}dra \cite[Theorem 1.6]{1105.0825}.

\begin{theorem}[{\cite[Theorem 1.6]{1105.0825}}]%\label{thm:app_meas_pres}
  Let $X$ be a compact space and $\mu$ a Borel probability measure on $X$.
  Assume that the class $a \in \HHH^1(X;\A)$ admits a standard representative.
  Let $\Gm$ be the group of $\mu$-preserving homeomorphisms in $\G$.
  If a homomorphism $\phi \colon K \to \G$, then the class $\phi^*[\mathfrak{G}_{x, \alpha}]$ is equal to zero in $\HHH^2(K;\RR)$.
  %In particular, it is trivial on amenable subgroups of $\G$.
\end{theorem}
\end{comment}

\begin{proof}[Proof of Theorem $\ref{thm:app_meas_pres}$]
  Let $\alpha$ be a standard representative of $a$.
  By Remark \ref{rem:existence_trans_numbers}, the mean translation number $\trot_{\mu, \alpha}$ is defined on the preimage $p^{-1}(\Gm)$.
  By Proposition \ref{prop:trans_number_hom}, $\trot_{\mu, \alpha}$ is a homomorphism on $p^{-1}(\Gm)$.
  Since $p^{-1}(\phi(K)) \subset p^{-1}(\Gm)$, Theorem \ref{thm_appendix} implies Theorem \ref{thm:app_meas_pres}.
\end{proof}

\begin{remark}\label{rem:integral_not_vanish}
  In general, Theorems \ref{thm_appendix} and \ref{thm:app_meas_pres} do not hold in $H_{\grp}^2(\Gm;\ZZ)$.
  Indeed, if $X = S^1$ and $a \in H^1(X;\ZZ)$ is the generator, the class $[\mathfrak{G}_{x, \alpha_{\ZZ}}] \in H_{\grp}^2(\Homeo_+(S^1);\ZZ)$ is equal to the Euler class of $\Homeo_+(S^1)$, which is non-zero on the Lebesgue measure preserving subgroup.
  For example, it is known that the Euler class with coefficients in $\ZZ$ is non-zero on the subgroup $SO(2)$ of $\Homeo_+(S^1)$, which preserves the Lebesgue measure.
  \end{remark}

\subsection{Cohomological non-triviality of the Gal--K\k{e}dra cocycle}
Concerning the non-triviality of the cohomology class $[\mathfrak{G}_{x, \alpha}]$, Gal--K\k{e}dra showed in \cite{1105.0825} the following:

\begin{theorem}[{\cite[Theorem 1.9]{1105.0825}}]\label{thm:app_GK_non_zero}
  Let $K$ be a connected topological group and $\phi \colon K \to \G$ be a homomorphism.
  Assume that the canonical homomorphism $\HHH^2(BK;\A) \to \HHH^2(BK^{\delta};\A)$ is injective, where $BK$ and $BK^{\delta}$ are the classifying spaces.
  Then the class $\phi^*[\mathfrak{G}_{x, \alpha}]$ is non-zero if and only if $\ev_x^*a \in \HHH_{\grp}^1(K;\A)$ is non-zero.
  Here $\ev_x \colon K \to X$ is the map defined by $\ev_x(k) = (\phi(k))(x) \in X$.
\end{theorem}

By using this, we show that the Gal--K\k{e}dra cocycle for several Seifert-fibered $3$-manifolds is cohomologically non-trivial (we refer to \cite{MR741334} for basics of Seifert-fibered $3$-manifolds).

\begin{theorem}\label{thm:app_seifert}
  Let $X$ be a closed Seifert-fibered $3$-manifold whose Euler number is equal to zero.
  If $X$ is not covered by $S^3$, there exists a cohomology class $a \in \HHH^1(X;\A)$ such that the cohomology class $[\mathfrak{G}_{x, \alpha}] \in \HHH_{\grp}^2(\G;\A)$ is non-zero.
\end{theorem}

\begin{proof}
  Let $K$ be the identity component $\Homeo_0(X)$ of the homeomorphism group $\Homeo(X)$ and $\phi \colon K \to \G$ the inclusion.
  Thurston's theorem \cite{thurston74} asserts that the map
  \[
    B\iota^* \colon \HHH^*(BK;\A) \to \HHH^*(BK^{\delta};\A)
  \]
  is an isomorphism.
  Hence, by Theorem \ref{thm:app_GK_non_zero}, it suffices to construct a cohomology class $a \in \HHH^1(X;\A)$ such that $\ev_x^*a \in \HHH^1(K;\A)$ is non-zero for some point $x \in X$.
  Since $X$ is a Seifert-fibered $3$-manifold, $X$ is equipped with the circle action
  \[
    \iota \colon S^1 \to K
  \]
  whose orbits are the fibers of $X$.
  Hence, we now construct a cohomology class $a \in \HHH^1(K;\A)$ which is non-zero on a fiber of $X$, that is, $\iota^* \ev_x^*a \neq 0$.

  %By Corollary \ref{cor:non-zero}, it suffices to show the existence of a cohomology class $a \in \HHH^1(X;\A)$ which is non-zero on a fiber.
  Let $(g;(\alpha_1, \beta_1), \dots, (\alpha_n, \beta_n))$ be a Seifert invariant of $X$.
  If $g\geq 0$, a presentation of the fundamental group $\pi_1(X)$ is given by
  \begin{align*}
    \langle a_i, b_i, q_j, h \mid [h, a_i] = [h, b_i] = [h, q_j] = 1, \ q_j^{\alpha_j}\HHH^{\beta_j} = 1, \ q_1\dots q_n [a_1, b_1]\dots [a_g, b_g] = 1\rangle.
  \end{align*}
  Here $h$ corresponds to a regular fiber of $X$ and $1$ denotes the unit element of $\pi_1(X)$.
  Note that $h$ is an element of infinite order since $X$ is not covered by $S^3$ (see \cite[Lemma 3.2]{MR705527} for example).
  Hence $\alpha_j$ is non-zero for any $j$.
  We set $\alpha = \alpha_1 \cdots \alpha_n$ and define a map $\phi \colon \pi_1(X) \to \A$ by
  \[
    \phi(a_i) = \phi(b_i) = 0, \ \  \phi(q_j) = \frac{\alpha \cdot \beta_j}{\alpha_j}, \ \  \phi(h) = \alpha.
  \]
  This map is a well-defined homomorphism since the Euler number of $X$ is equal to zero.
  %Moreover, the restriction of $\phi$ to the cyclic subgroup $\langle h \rangle$ is non-zero.
  Since $\phi(h) = \alpha \neq 0$, the homomorphism $\phi$ defines a desired cohomology class $a \in \HHH^1(X;\A)$.
  %Then, by using the well-known presentation of $\pi_1(X)$ and the assumption on the Euler number, we can construct a homomorphism $\pi_1(X) \to \A$ whose restriction to $\pi_1(C)$ is non-zero.

  The proof for the case where $g < 0$ is similar.
\end{proof}

\section{Undistortion elements in the group of bundle automorphisms}\label{sec:appl2}

In this section, we assume that $S_a$ is non-empty.
We apply the translation numbers to study distortion in the group $\hG$.

\subsection{Distortion in groups}\label{subsec:distortion}
Let us recall distortion in groups.
Let $\GG$ be a finitely generated group and $S$ a finite (symmetric) generating set of $\GG$.
For any element $\gg \in \GG$, its \textit{word norm} $| \gg |_S$ is defined by
\[
  | \gg |_S = \min \{ l \colon \gg = s_1 \cdots s_l \text{ for some } s_j \in S \},
\]
and its \textit{translation length} $\tau(\gg)$ is given as
\[
  \tau(\gg) = \lim_{n \to \infty} \frac{| \gg^n |_{S}}{n}.
\]
An element $\gg$ is called \textit{distorted} in $\GG$ if $\tau(\gg) = 0$ and \textit{undistorted} otherwise.
In contrast to the word norm and the translation length, distortion and undistortion do not depend on the choice of generators.

For an arbitrary group $G$, an element $g \in G$ is \textit{undistorted} in $G$ if it is undistorted in every finitely generated subgroup of $G$.

%For a homomorphism $f \colon G \to H$ between groups, an element $g$ of $G$ is undistorted in $G$ provided that $f(g) \in H$ is undistorted in $H$.

\begin{comment}
In this paper, we study undistortion in a certain cenctral extension.
The Euler class $e(E)$ is called \textit{bounded} if there exists a cocycle representing $e(E)$ whose image is a bounded set in $Z$.
For a central extension $0 \to Z \to E \to G \to 1$, any non-zero element in $Z$ is undistorted provided that the Euler class $e(E)$ is bounded.
In fact, if the Euler class of a central extension of finitely generated groups is bounded, the central extension is quasi-isometrically trivial (\cite{MR1189238v5} \cite{MR1189238}).
Here the central extension is quasi-isometrically trivial if there exists a quasi-isometry $f \colon E \to Z \times G$ such that the following diagram commutes up to bounded error:
\[
\xymatrix{
E \ar[r]^-{f} \ar[d]^-{\pi} & Z \times G \ar[d]^-{\pi_2} \\
G \ar@{=}[r] & G,
}
\]
where $\pi_2$ is the second projection.
\end{comment}

\subsection{A seminorm on $\hG$}\label{subsec:seminorm}

From here to the end of the next subsection, we fix a standard cocycle $\alpha$ of $a$.
By the definition of standard cocycles, the value $\rho_{x, \alpha}(\hg)$ depends continuously on $x$.
Hence, the set $\{ \rho_{x, \alpha}(\hg) \colon x \in X \}$ is bounded when $X$ is compact.
We set
\begin{align}\label{seminorm}
  \| \hg \|_{\alpha} = \sup_{x \in X} \left| \rho_{x, \alpha}(\hg) \right|.
\end{align}

\begin{lemma}
  If $X$ is compact, then $\| \cdot \|_{\alpha}$ is a seminorm on the group $\hG$.%, that is, it is symmetric and satisfies the triangle inequality.
\end{lemma}

\begin{proof}
  For any $\hg \in \hG$, we have $\rho_{x, \alpha}(\hg^{-1}) = -\rho_{g^{-1}(x), \theta}(\hg)$, where $g = p(\hg) \in \G$.
  Hence the equality $\| \hg^{-1} \|_{\alpha} = \| \hg \|_{\alpha}$ follows.

  For $\hg, \hh \in \hG$, we set $g = p(\hg)$ and $h \in p(\hh)$.
  Then we obtain
  \begin{align*}
    \rho_{x, \alpha}(\hg\hh) = \rho_{h(x), \theta}(\hg) + \rho_{x, \alpha}(\hh),
  \end{align*}
  and this implies the inequality $\| \hg \hh \|_{\alpha} \leq \| \hg \|_{\alpha} + \| \hh \|_{\alpha}$.
\end{proof}

\subsection{Undistortion in $\hG$}
%In this section, we assume the space $X$ is compact unless otherwise stated.
Let $\GG$ be a finitely generated subgroup of $\hG$ and $S$ a finite generating set.
Let $| \cdot |_{S}$ be the word norm with respect to $S$.
It is easily verified that an inequality
\begin{align}
  C \cdot | \hg |_{S} \geq \| \hg \|_{\alpha}
\end{align}
holds for any $\hg \in \GG$, where $C = \max \{ \| s \|_{\alpha} \colon s \in S \}$.

\begin{proof}[Proof of Theorem $\ref{thm:distortion}$]
  Note that the inequality $|\trot_{\mu, \alpha}(\hg)| > 0$ implies the existence of a point $x \in X$ with $|\trot_{x, \alpha}(\hg)| > 0$ by Remark \ref{rem:mean_t} (1).
  Hence it suffices to show that $\hg$ is undistorted in $\hG$ when $|\trot_{x, \alpha}(\hg)| > 0$.
  Let $\GG$ be a finitely generated subgroup of $\hG$ that contains $\hg$ and $S$ a finite generating set of $\GG$.
  Then we have
  \begin{align*}
    \frac{C \cdot | \hg^n |_S}{n} \geq \frac{\| \hg^n \|_{\alpha}}{n} \geq \frac{|\rho_{x, \alpha}(\hg^n)|}{n}.
  \end{align*}
  Hence we obtain
  \begin{align*}
    \tau(\hg) \geq \frac{1}{C}\cdot |\trot_{x, \alpha}(\hg)| > 0 % \ \  \text{ or } \ \  \tau(\hg) \geq \frac{1}{C}\cdot |\trot_{\mu, \alpha}(\hg)| > 0
  \end{align*}
  by assumption.
  Since $\GG$ is arbitrary, the element $\hg$ is undistorted in $\hG$.
\end{proof}

\begin{comment}
\begin{remark}
  In \cite{MR2854098}, Gal and K\k{e}dra studied distortion in $\G = \Homeo(X,a)$ by using the equivalent measures and the local rotation number.
  Theorem \ref{thm:distortion} can be seen as a $\hG$-analogue of their results \cite[Theorems 1.1 and 1.6]{MR2854098}.
\end{remark}
\end{comment}

\begin{comment}
\begin{remark}
  \begin{enumerate}[$(1)$]
    \item The local translation number $\trot_{x, \alpha}(\hg)$ is non-zero if and only if the function $\rho_{x, \alpha}$ is unbounded on the cyclic subgroup $\langle \hg \rangle$.
    \item As stated in \cite[Remark 3.6]{MR2854098}, a method often used to prove an element $g$ is undistorted in $G$ is to construct a homogeneous quasimorphism $\phi \colon G \to \RR$ with $\phi(g) \neq 0$.
    It is in general difficult to construct such a homogeneous quasimorphism.
    The local translation number $\trot_{x, \alpha}$ is also not a homogeneous quasimorphism on the whole group $\hG$ in general.
    However, when we use Theorem \ref{thm:distortion} to show the undistortedness of $\hg$, it is enough to check the quasimorphism $\rho_{x, \alpha}$ defined on the cyclic subgroup $\hg$ is unbounded.
  \end{enumerate}
\end{remark}
\end{comment}

\begin{corollary}\label{cor:undistorted}
  Let $X$ be a compact space, $\hg$ an element of $\hG$, and set $g = p(\hg) \in \G$.
  Assume that $g$ has a periodic point $x \in X$ of period $q$.
  If there exists a lift $\hx \in \hX$ of $x$ such that $\hg^q(\hx) \neq \hx$, then $\hg$ is undistorted in $\hG$.
  In particular, the element $T_r$ is undistorted in $\hG$ for any non-zero number $r \in \A$.
\end{corollary}

\begin{proof}
  By assumption, the points $\hg^q(\hx)$ and $\hx$ belong to the same fiber, and hence there exits a non-zero number $r \in \A$ such that $\hg^q(\hx) = T_r(\hx)$.
  Then, by the calculation same as in the proof of Proposition \ref{prop:periodic}, we have
  \[
    \trot_{x, \alpha}(\hg) = \frac{r}{q} \neq 0.
  \]
  This, together with Theorem \ref{thm:distortion}, implies the corollary.
\end{proof}

\section*{Acknowledgement}
The author would like to thank Yoshifumi Matsuda for helpful discussions.
The author is supported by JSPS KAKENHI Grant Number JP21J11199.

\bibliographystyle{amsalpha}
\bibliography{references.bib}
\end{document}